\documentclass[11pt]{amsart}     
\usepackage[english]{babel}

\usepackage[foot]{amsaddr}
\usepackage{amsfonts}
\usepackage{hyperref}
\usepackage{mathrsfs}
\usepackage{amsthm}
\usepackage{amsmath}
\usepackage{amssymb}
\usepackage{graphicx}
\usepackage{tikz-cd}
\usepackage[curve]{xypic}
\usepackage{verbatim}

\theoremstyle{plain}
\newtheorem{Th}{Theorem}[section]
\newtheorem{Le}[Th]{Lemma}
\newtheorem{Pro}[Th]{Proposition} 
\newtheorem{Cor}[Th]{Corollary}

\theoremstyle{definition}
\newtheorem{Def}[Th]{Definition}

\numberwithin{equation}{section}

\def\eI{{^{I}\!\mathrm{E}}}
\def\eII{{^{I\!I}\!\mathrm{E}}}

\def \Hn{\mathbb{H}^n}
\def\dH{\partial\mathbb{H}^n}
\def\acts{\curvearrowright}

\begin{document}
\title[]{The boundary model for the continuous cohomology of Isom$^+(\Hn)$.}
\author[]{Hester Pieters}
\address{Mathematics department, Technion, Haifa 32000, Israel}
\email{hester.frede@technion.ac.il}
\thanks{This research was supported by Swiss National Science Foundation grant number PP00P2-128309/1.}
\date{}
\maketitle

\begin{abstract}
We prove that the continuous cohomology of $\text{Isom}^+(\mathbb{H}^n)$ can be measurably realized on the boundary of hyperbolic space. This implies in particular that for $\text{Isom}^+(\mathbb{H}^n)$ the comparison map from continuous bounded cohomology to continuous cohomology is injective in degree $3$. We furthermore prove a stability result for the continuous bounded cohomology of $\text{Isom}(\mathbb{H}^n)$ and $\text{Isom}(\mathbb{H}_{\mathbb{C}}^n)$.  
\end{abstract}

\section{Introduction}
\label{intro}
Let $H$ be a locally compact second countable group. We denote by $H_m^*(H;\mathbb{R})$ the measurable cohomology of $H$ with trivial real coefficients (see \cite{Mo12}, \cite{Mo3}, \cite{Mo4} and Section \ref{Prelim}). If H acts on a measure space $X$ we define $H_m^*(H\acts X;\mathbb{R})$ to be the cohomology of the cocomplex $(C^*(X;\mathbb{R})^H, \delta)$, where $C^p(X;\mathbb{R})^H$ is the space of $H$-invariant measurable maps $X^{p+1}\to\mathbb{R}$ identifying those which agree almost everywhere and $\delta$ is the standard homogeneous differential (see Section \ref{Prelim} for precise definitions). Fix a basepoint $x\in X$. Given a cocycle $\alpha\in C^p(X;\mathbb{R})^H$ we obtain a cocycle $\alpha_x\in C(H^{p+1};\mathbb{R})^H$ by
\[
\alpha_x(h_0,\dots,h_p):=\alpha(h_0\cdot x,\dots, h_p\cdot x).
\]
Since the class of $\alpha_x$ does not depend on the chosen basepoint this defines a map
\[
\iota_X:H_m^*(H\acts X;\mathbb{R})\to H_m^*(H;\mathbb{R}).
\]
We say that the measurable cohomology of $H$ (with trivial $\mathbb{R}$-coefficients) is \textit{measurably realized on $X$} if this map is an isomorphism. Since measurable and continuous cohomology coincide for trivial $\mathbb{R}$-coefficients \cite[Theorem A]{AM} we will then also say that the continuous cohomology of $H$ is measurably realized on $X$. Let $\text{Isom}^+(\mathbb{H}^n)$ be the group of orientation preserving isometries of real hyperbolic $n$-space $\mathbb{H}^n$. Our main result is 
\begin{Th}
\label{iso}
Let $n\geq 2$. The continuous cohomology of $\mathrm{Isom}^+(\mathbb{H}^n)$ is measurably realized on the boundary $\partial\mathbb{H}^n$ of real hyperbolic space (with the natural action). Furthermore, all cohomology classes have essentially bounded representatives in this boundary resolution. 
\end{Th}
An immediate consequence of Theorem \ref{iso} is
\begin{Cor}
\label{inj}
The comparison map 
\[
c: H^3_{c,b}(\text{Isom}^+(\mathbb{H}^n);\mathbb{R})\to H_c^3(\text{Isom}^+(\mathbb{H}^n);\mathbb{R})
\]
is injective for all $n\geq 2$.
\end{Cor}
For $n=3$ the above theorem is a result of Bloch. Recall that, by the van Est isomorphism, $H_c^*(\text{Isom}^+(\mathbb{H}^3);\mathbb{R})$ can be identified with the de Rham cohomology of the $3$-sphere. Thus $H_c^3(\text{Isom}^+(\mathbb{H}^3);\mathbb{R})$ is one dimensional and it is well known that it is generated by the volume function $\text{Vol}\in L^\infty((\partial\mathbb{H}^3)^4;\mathbb{R})^{\text{Isom}^+(\mathbb{H}^3)}$ which sends four points in the boundary to the volume of the ideal simplex they span. Hence Bloch's result implies that, up to scalar multiplication, Vol is the only cocycle in degree 3 defined on the boundary. He used this to show that the Bloch-Wigner dilogarithm is essentially the only measurable function on $\mathbb{P}^1_{\mathbb{C}}$ that satisfies the five term relation. Indeed, applying such a function to the cross ratio of $4$ points in $\partial\mathbb{H}^3$ gives a measurable cocycle and thus a multiple of the volume function. The main difficulty in the generalization from $n=3$ to higher dimensions comes from the fact that  for $n\geq 4$ the stabilizer of $3$ points in $\partial\mathbb{H}^n$ is not trivial. This prevents a straightforward generalization of Bloch's proof for degree $p>3$. Below we describe our strategy of proof.

\subsection*{About the proof of Theorem \ref{iso}}
Let $G=\text{Isom}^+(\Hn)$ and $C^p=C((\dH)^{p+1};\mathbb{R})$ the $G$-module of measurable functions $(\dH)^{p+1}\to\mathbb{R}$ identifying those which are equal almost everywhere. We follow Bloch's approach by looking at the spectral sequences associated to the double complex $(C(G^{q+1};C^p)^G,d,\delta)$ (with differentials $d,\delta$ to be defined below in Section 3). For $p>2$ we have $\eI_1^{p,q}=H^*_m(\mbox{SO}(n-2);C^{p-3})$, with SO$(n-2)$ the stabilizer of $3$ points in the boundary of hyperbolic space. In the case of $n=3$ this group is trivial and thus $\eI_1^{p,q}$ automatically vanishes for $q>0$. Because of this, in Bloch's proof for $n=3$ the fact that the spectral sequence degenerates at the second page already follows from looking at the first page. One would expect that $H^*_m(K;A)$ vanishes if $K$ is compact no matter what the coefficients $A$ are. However, if the coefficients $A$ are not locally convex there is no way to integrate over them so it is not possible to construct a coboundary by integration over the first variable. Since $C((\partial\mathbb{H}^n)^{p-2};\mathbb{R})$ is a (non-locally convex) F-space it is not clear (how to prove) that $H^*_m(\mbox{SO}(n-2);C^{p-3})$ vanishes for $n>3$.  Here we will instead prove that if a cocycle $[\alpha]\in \eI_1^{p,q}$ survives to the second page of the spectral sequence, i.e. if $\delta\alpha=d\lambda$, then it is cohomologous in $\eI_2^{p,q}$ to a coboundary in ${}\eI_1^{p,q}$. As in the proof in \cite{AM} of the isomorphism between measurable and continuous cohomology in the case of Fr\'echet coefficients \cite[Theorem A]{AM}, we would like to show by dimension-shifting induction that each cohomology class has a representative that is locally totally bounded. In fact, by a double induction argument, we will show that this is the case for a cocycle in $\eI_2^{p,q}$. The first step of this dimension-shifting induction argument (Proposition \ref{bounded}) implies in particular that all cocycles in $(C^p)^G$ are cohomologous to essentially bounded cocycles, which is the second part of Theorem \ref{iso}. 
It is not known if Theorem \ref{iso} generalizes to all semisimple Lie groups $G$ with Furstenberg boundary $G/P$. Examples indicate that this might be the case. In \cite{Go1}, \cite{Go2} Goncharov defines measurable cocycles on the space of flags $\mathcal{F}l(\mathbb{C}^m)$ representing the Borel classes in $H_c^{2n-1}(\mathrm{SL}_m(\mathbb{C});\mathbb{R})$ for $n=2,3$ and $m\geq 2n-1$ using the classical di- and trilogarithm. These cocycles are all bounded and therefore give further evidence for the conjecture (\cite{Du}, \cite{Mon2}) that the comparison map between continuous bounded cohomology and continuous cohomology is an isomorphism for all semisimple connected Lie groups with finite center. This has so far only been established in a few cases. For degree 2 it was proven by Burger and Monod in \cite{BM}. In degree 3 and 4 it has been proven for $\mathrm{SL}_2(\mathbb{R})$ by Burger-Monod \cite{BM2} and Hartnick-Ott \cite{HO} respectively. For $\text{Isom}^+(\mathbb{H}^n)$, injectivity in degree $3$ was so far only known for $n=2$ by Burger-Monod \cite{BM} and for $n=3$ by Bloch's result. 
\\\\
Corollary \ref{inj} also follows from a simpler argument which only uses some basic properties of hyperbolic space and the injectivity in degree 3 for $\text{Isom}^+(\mathbb{H}^3)$. This method of proof also gives a stability result for the isometry group $\text{Isom}(\mathbb{H}^n_{(\mathbb{C})})$ of real (or complex) hyperbolic space.
\begin{Th}
\label{stability}
If $k+1\leq n$ then there exists an injection 
\[
H_{c,b}^k(\mathrm{Isom}(\mathbb{H}_{(\mathbb{C})}^n);\mathbb{R})\hookrightarrow H_{c,b}^k(\mathrm{Isom}(\mathbb{H}_{(\mathbb{C})}^{n+1});\mathbb{R}).
\] 
\end{Th}
Such stability results give a further tool for computing the continuous bounded cohomology of Lie groups. In \cite{Mon3} Monod proves  a similar result for the continuous bounded cohomology of $\mathrm{SL}_n$. More precisely, for any local field $k$ and $0\leq q\leq n-1$ he shows that the standard embedding $\mathrm{GL}_{n-1}(k)\hookrightarrow \mathrm{SL}_n(k)$ induces an isomorphism $H^q_{c,b}(\mathrm{SL}_n(k))\cong H^q_{c,b}(\mathrm{GL}_{n-1}(k))$. He proves this using nontrivial coefficients of $L^\infty$ type and a spectral sequence argument.

\subsection*{Acknowledgement}
The author would like to thank her doctoral advisor Michelle Bucher for many useful discussions and for carefully reading multiple drafts of this paper, Tobias Hartnick for his interest and helpful suggestions and the referee for suggesting many improvements to the presentation. 
\section{Preliminaries}
\label{Prelim}
\subsection{Polish Abelian $H$-modules}
Let $H$ be a locally compact second countable group. A \textit{Polish space} is a separable completely metrizable topological space, a \textit{Polish group} is a topological group which is also a Polish space in its topology and a \textit{Polish Abelian $H$-module} is a triple $(A,\rho,T)$, where $A$ is a Polish Abelian group with a translation-invariant metric $\rho$ compatible with the topology on $A$ and $T:H\curvearrowright A$ an action by continuous automorphisms. It is a \textit{F-space} if $A$ is furthermore a separable real topological vector space in its Polish topology. A \textit{Fr\'echet space} is a locally convex F-space. 

\begin{Def}
A Borel map $f:X\to Y$ from a locally compact space to a Polish space is \textit{locally totally bounded} if for any compact subset $K\subset X$ the image $f(K)$ is precompact in $Y$. 
\end{Def}

It follows from the Borel selection theorem (see e.g. \cite[Section 423]{Fre} and \cite[Lemma 32]{AM}) that 

\begin{Le}
\label{Borel section}
 Let $A$ be a Polish Abelian group and let $B<A$ be a closed subgroup. Then there exists a section $s:A/B\to A$ of the natural projection map $p:A\to A/B$ that is Borel and locally totally bounded. If $A$ is furthermore a Fr\'echet space then there exists a section that is continuous.
 \end{Le}

For any $\sigma$-finite measure space $(X,\mathcal{B},\mu)$ and any separable metric space $A$ let $C(X;A)$ be the set of equivalence classes of measurable functions $X\to A$, identifying those which agree $\mu$-almost everywhere (a.e.). Any class of measurable functions contains a Borel function so we could also define $C(X;A)$ as consisting of equivalence classes of Borel maps $X\to A$ and obtain the same space. We endow this space with the topology of convergence in measure. When $A=\mathbb{R}$ this gives the usual F-space structure on $C(X;\mathbb{R})$ and this space is often denoted by $L^0(X)$. There holds an exponential law for these spaces, i.e. if $(X_1,\mathcal{B}_1,\mu_1)$ and $(X_2,\mathcal{B}_2,\mu_2)$ are two $\sigma$-finite measure spaces then \cite[Theorem 1]{Mo3}:
\begin{equation}
\label{explaw}
C(X_1\times X_2;A)\cong C(X_1;C(X_2;A))\cong C(X_2;C(X_1;A)).
\end{equation}
If $A$ is a Polish Abelian $H$-module and $(X,\mathcal{B},\mu)$ is a $\sigma$-finite standard Borel space on which $H$ acts as a Borel transformation group such that it leaves some finite measure $\nu$ that is equivalent to $\mu$ quasi-invariant  then $C(X;A)$ is again a Polish Abelian $H$-module with action given by 
\[
(h\cdot f)(x_1,\dots,x_p):=  h\cdot (f(h^{-1}x_1,\dots, h^{-1}x_p)), 
\] 
for $f\in C(X;A)$, $h\in H$ and $x_1,\dots,x_p\in X$ \cite[Proposition 12]{Mo3}. Note that any equality between elements of $C(X;A)$ is understood to hold $\mu$-a.e.
\subsection{Cohomology theories}
 For locally compact second countable groups $H$ and Polish Abelian $H$-modules $A$ in \cite{Mo12}, \cite{Mo3} and \cite{Mo4} C.C. Moore developed the measurable cohomology theory $H^*_m(H;A)$. Denote by $C(H;A)^H$ the submodule of $C(H;A)$ consisting of $H$-invariant maps and let 
\[
d:C(H^{p+1};A)^H\to C(H^{p+2};A)^H
\]
be the standard homogeneous coboundary operator, i.e. for a cochain $\alpha\in C(H^{p+1};A)^H$ and $h_0,\dots,h_{p+1}\in H$
\begin{eqnarray*}
d\alpha(h_0,\dots,h_{p+1}):= \sum_{i=0}^{p+1} (-1)^i \alpha(h_0,\dots,\hat{h}_i,\dots,h_{p+1}).
\end{eqnarray*}
The \textit{measurable cohomology groups for $H$ with coefficients in $A$} are
\[
H_m^p(H;A):=\frac{\ker(d:C(H^{p+1};A)^H\to C(H^{p+2};A)^H)}{\text{im}(d:C(H^p;A)^H\to C(H^{p+1};A)^H)}.
\]
Analogously, $C_c(H^{p+1};A)^H$ and $C_{c,b}(H^{p+1};A)^H$ are defined to be the $H$-invariant cochains $H^{p+1}\to A$ that are respectively continuous and continuous bounded and we obtain the \textit{continuous cohomology groups}
\[
H_c^p(H;A):=\frac{\ker(d:C_c(H^{p+1};A)^H\to C_c(H^{p+2};A)^H)}{\text{im}(d:C_c(H^p;A)^H\to C_c(H^{p+1};A)^H)},
\]
and the \textit{continuous bounded cohomology groups}
\[
H_{c,b}^p(H;A):=\frac{\ker(d:C_{c,b}(H^{p+1};A)^H\to C_{c,b}(H^{p+2};A)^H)}{\text{im}(d:C_{c,b}(H^p;A)^H\to C_{c,b}(H^{p+1};A)^H)},
\]
where the appropriate coefficients $A$ for continuous bounded cohomology are the dual of a separable Banach space on which $H$ acts continuously and by linear isometries. For more information about this technical requirement see \cite{Mon}. For continuous cohomology we allow as coefficients $A$ all Fr\'echet spaces with a continuous $H$-action.  
\subsection{Buchsbaum's criterion}
\label{BuchsbaumSubsection}
Denote by $P(H)$ the category of Polish Abelian $H$-modules. A short exact sequence
\[
\xymatrix{
 0\ar[r] & A \ar[r]^{i}& B \ar[r]^j &C \ar[r] &0
 }
 \]
 in $P(H)$ is exact algebraically and such that the maps $i$ and $j$ are continuous homomorphisms intertwining the action of $H$. An \textit{effaceable cohomological functor $H^*(H;\cdot)$ on $P(H)$} is a a family $H^n(H;\cdot), n\geq 0$ of covariant functors from $P(H)$ to the category of Abelian groups such that the following three conditions hold
\begin{enumerate}
\item $H^0(H;A)=A^H$
\item \label{cond2} Every short exact sequence $0\rightarrow A\rightarrow B\rightarrow C\rightarrow 0$ of Polish Abelian $H$-modules induces a long exact sequence in cohomology
\begin{eqnarray*}
&& 0\to H^0(H;A)\to H^0(H;B)\to H^0(H;C)\to H^1(H;A)\to \dots\nonumber \\
&&\dots\to H^k(H;B)\to H^k(H;C)\to H^{k+1}(H;A)\to\dots,
\end{eqnarray*}

\item  \label{cond3} $H^*(H;\cdot)$ is \textit{effaceable} in the category of Polish Abelian $H$-modules. That is, for any Polish Abelian $H$-module $A$ and any  $a\in H^k(H;A)$ there exists a short exact sequence 
\[
0\to A\to B\to C\to 0
\]
such that the image of $a$ in $H^k(H;B)$ vanishes. 
\end{enumerate}
By Buchsbaum's criterion (\cite{Bu}), an effaceable cohomological functor $H^*(H;\cdot)$ on $P(H)$ is unique (if it exists). More general,

\begin{Le}[Buchsbaum's criterion, \cite{Bu}]
\label{BuchsbaumLemma}
Let $H_1^n(H;\cdot)$ and $H_2^n(H;\cdot)$ be families of covariant functors from  an Abelian category $\mathcal{A}$ to the category of Abelian groups such that $H_1^0(H,A)=H_2^0(H;A)$ for all $A\in\mathcal{A}$ and such that $H_1^n(H;\cdot)$ and $H_2^n(H;\cdot)$ both satisfy the above conditions \eqref{cond2} and \eqref{cond3} for the category $\mathcal{A}$. Then  
\[
H_1^n(H;A)\cong H_2^n(H;A),
\]
for all $n\in\mathbb{N}$ and all $A\in\mathcal{A}$. 
\end{Le}

C.C. Moore proved that measurable cohomology satisfies the above three requirements in the category $P(H)$ and is therefore the unique effaceable cohomological functor on $P(H)$ \cite[Section 4]{Mo3}. Let us briefly discuss why the three above conditions hold for the measurable cohomology $H_m^*$. The fact that $H_m^0(H;A)=A^H$ is immediate from the definition. For the third condition, Moore proves that $H^p(H;C(H;A))=0$ for $p>0$ \cite[Theorem 4]{Mo3}. Then any cohomology class $[\alpha]\in H_m^k(H;A)$ is effaced by the inclusion $\iota:A\hookrightarrow C(H;A)$ (sending $a\in A$ to the constant function $\equiv a$). Indeed, $\iota^*(\alpha)=\delta\beta$ with  $\beta:H^p\to C(H;A)$ defined by 
\[
\beta(h_1,\dots,h_p)(h):=(-1)^p\alpha(h_1,\dots,h_p,h).
\]
Lastly, the second condition, i.e. the existence of long exact sequences, is ensured by the existence of Borel sections: Let
\[
\xymatrix{
 0\ar[r] & A \ar[r]^{i}& B \ar[r]^j &C \ar[r] &0
 }
\]
be a short exact sequence of Polish Abelian $H$-modules. The induced sequence
\[\xymatrix{
0\ar[r] & C^q(H;A) \ar[r]^{i^*} & C^q(H;B)  \ar[r]^{j^*} & C^q(H;C) \ar[r] & 0
}
\]
is then also exact: The map $i^*:C^q(H;A)\to C^q(H;B)$ is clearly injective and the induced maps $i^*$ and $j^*$ are continuous. Furthermore, by Lemma \ref{Borel section} there is a Borel map $s:C\to B$ with $j\circ s=\mathrm{id}_C$ and thus if $\alpha\in C^q(H;C)$ then $s^*\circ \alpha\in C^q(H;B)$ is mapped to $\alpha$. Hence $j^*$ is surjective and the sequence is exact. Then a long exact sequence as in the second condition can be constructed in the standard way. 
\\\\
Since in general there exists no continuous cross section $B/A\to B$ continuous cohomology has no long exact sequences when we allow all Polish $H$-modules as coefficients. However, when restricting to Fr\'echet modules there do exist continuous cross sections and continuous cohomology is the unique effaceable cohomological functor on this category. In \cite{AM} T. Austin and C.C. Moore prove that measurable cohomology is also effaceable when restricted to Fr\'echet modules. Hence for Fr\'echet coefficients $A$ the natural inclusion $C_c^*(H;A)\hookrightarrow C^*(H;A)$ induces an isomorphism $H_m^*(H;A)\cong H_c^*(H;A)$ \cite[Theorem A]{AM}. The crucial step in the proof is \cite[Proposition 33]{AM} which states that for any cohomology class in $H_m^*(H;A)$ there exists a locally totally bounded representative. The main ingredients of its proof are dimension shifting and Lemma \ref{Borel section} which ensures that a locally totally bounded measurable cocycle $\bar{\alpha}:H^{p}\to C(H;A)/\iota(A)$ can be lifted to a locally totally bounded measurable map $\alpha: H^{p}\to C(H;A)$.

\subsection{Dimension shifting}
Let $\iota:A\hookrightarrow C(H;A)$ be the embedding of the $H$-module $A$ into $C(H;A)$ as the closed submodule of constant maps. Then, since $H^p(H;C(H;A))=0$ for $p>0$ \cite[Theorem 4]{Mo3}, the short exact sequence
\[
0\to A\hookrightarrow C(H;A) \twoheadrightarrow C(H;A)/\iota(A)\to 0
\]
induces isomorphisms 
\begin{equation}
\label{dimensionshifting}
H^p_m(H;A)\cong H_m^{p-1}(H;C(H;A)/\iota(A)),
\end{equation}
for all $p>0$. This allows for the technique of \textit{dimension shifting}, that is we can rewrite a cohomology group as a cohomology group of lower degree (but with different coefficients). Then by induction on degree, some algebraic properties that clearly hold in lower degree may be shown to hold in higher degrees as well. We will use this technique in the proof of Theorem \ref{iso}. 
\\\\
The connecting map $H^p_m(H;A)\to H_m^{p-1}(H;C(H;A)/\iota(A))$ realizing the above isomorphism \eqref{dimensionshifting} is induced by the map
\[
Q:C^p(H;A)\to C^{p-1}(H;C(H;A))
\]
given by 
\[
(Q\alpha)(h_0,\dots,h_{p-1})(h):=(-1)^p\alpha(h_0,\dots,h_{p-1},h),
\]
for $\alpha\in C^p(H;A)$ and $h,h_0,\dots,h_{p-1}\in H$. If $\alpha$ is a cocycle it follows directly that $d(Q\alpha)(h_0,\dots,h_p)$ is the constant map $h\mapsto \alpha(h_0,\dots,h_p)$ and thus the image of $Q\alpha$ under the quotient map $C(H;A)\twoheadrightarrow C(H;A)/\iota(A)$ defines a class in $H_m^{p-1}(H;C(H;A)/\iota(A))$. Furthermore, it can be shown that this image only depends on the cohomology class of $\alpha$ and that $Q$ induces the connecting map.

\subsection{Eckmann-Shapiro Lemma}
\label{ES Lemma}
Let $L<H$ again be a closed subgroup and let $A$ be a Polish $L$-module. Furthermore, let Ind$_L^H(A)$ be the Polish $H$-module consisting of all equivalence classes of measurable maps $f:H\to A$ such that $f(hl)=l^{-1}\cdot f(h)\text{ for almost all pairs }(l,h)\in L\times H$ with the action of $H$ given by:
\[
(h\cdot f)(h')=f(h^{-1} h').
\]
There holds an Eckmann-Shapiro lemma \cite[Theorem 6]{Mo3}:
 \[
 H_m^q(H;\text{Ind}_L^H(A))\cong H_m^q(L;A).
\]
Let $s:H/L\to H$ be a locally totally bounded measurable section such that $s(L)=e$. For $f\in C(H/L;A)$ define $\bar{f}\in\text{Ind}_L^H(A)$ by
\[
\bar{f}(h)=h^{-1} \cdot s(hL) \cdot f(hL).
\]
This induces an isomorphism $C(H/L;A)\cong \text{Ind}_L^H(A)$ \cite[Proposition 17]{Mo3}. The action of $H$ on $C(H/L;A)$ becomes $(h\cdot F)(h'L)=\lambda(h,h'L)\cdot F(h^{-1}h'L)$, where
\[
\lambda (h,h'L):=s(h'L)^{-1}\cdot h\cdot s(h^{-1}h'L) \hspace{.2cm}\in L.
\] 
If cochains in $H^q_m(H;\text{Ind}_L^H(A))$ are continuous, as for example in degree $1$ (\cite[Theorem 3]{Mo3}), we can give explicit maps on the cochain level which induce the isomorphism of the Eckmann-Shapiro lemma. Define
\[
u^n:C_c(H^{q+1};A)^H\to C_c(L^{q+1};A)^L
\]
by
\[
u^n(\sigma)(l_0,\dots,l_q)=\sigma(l_0,\dots,l_q),
\]
for $l_0,\dots,l_q\in L$. Note that this is well defined because $\sigma$ is continuous in $H^{q+1}$. The map
\[
v^n: C_c(H^{q+1};C(H/L;A))^H\to C_c(H^{q+1};A)^L
\]
can be defined as follows: Let $\beta\in C_c(H^{q+1};C(H/L;A))$ and $h\in H$. Define $F_h(\beta) \in C_c(H^{q+1};A)$ by
\[
F_h(\beta)(h_0,\dots,h_q):= \beta(s(hL)h_0,\dots,s(hL)h_q)(hL),
\]
for $h_0,\dots, h_q\in G$. Then if $\beta$ is $H$-invariant, $F_h(\beta)$ is independent of $h$ and gives an element in $C_c(H^{q+1};A)^L$. We define $v^n(\beta)$ to be this element and then we can define
\begin{equation}
\label{phi}
\phi=u^n\circ v^n:C_c(H^{q+1}; C(H/L;A))^H\to C_c(L^{q+1};A)^L.
\end{equation}
Its inverse in cohomology is given by	
\begin{equation}
\label{psi}
\psi(\alpha)(h_0,\dots,h_q)(hL)=\alpha(\lambda(h_0,hL),\dots,\lambda(h_q,hL)),
\end{equation}
where $\alpha\in C_c(L^{p+1};A)^H$ and $h, h_0,\dots, h_q\in H$. 
\section{Proof of Theorem \ref{iso}}
Let $G=\text{Isom}^+(\Hn)$ and let $(K^{p,q}, d,\delta)$ be the first quadrant double complex defined by 
\[
K^{p,q}=C(G^{q+1};C^p)^G ,\text{ for } p,q\geq 0,
\]
where $C^p=C((\dH)^{p+1};\mathbb{R})$, the first differential $d:K^{p,q}\to K^{p,q+1}$ is the homogeneous coboundary operator and the second differential $\delta:K^{p,q}\to K^{p+1,q}$ is induced by the homogeneous coboundary operator $\delta:C^p\to C^{p+1}$. To such a complex one can associate two spectral sequences, $\eI_r^{p,q}$ and $\eII_r^{p,q}$, both converging to the cohomology of the total complex (see for example \cite[Chapter III, \textsection 14]{BT}).  If $\alpha\in C((\partial\mathbb{H}^n)^{p+1};\mathbb{R})^G$ is a cocycle then 
 \[
 \delta\alpha(y,x_0,\dots,x_p)=0,
 \]
 for almost all $(y,x_0,\dots,x_p)$. By the exponential law, there exists a $y\in\dH$ such that $\delta\alpha(y,x_0,\dots,x_p)=0$ for almost all $(x_0,\dots,x_p)$. Now define
\[
\beta(x_0,\dots,x_{p-1}):=\alpha(y,x_0,\dots,x_{p-1}).
\]
Then
\[
\delta\beta(x_0,\dots,x_p)=\sum_j (-1)^j\alpha(y,x_0,\dots, \hat{x}_j,\dots,x_p)=\alpha(x_0,\dots,x_p),
\]
for almost all $(x_0,\dots,x_p)$. Thus $(C^*,\delta)$ is an acyclic cocomplex and therefore
\begin{eqnarray*}
\eII_1^{p,q}=\begin{cases} C(G^{q+1};\mathbb{R})^G,&\text{if } p=0;\\
0, &\text{otherwise,}\end{cases}
\end{eqnarray*}
which implies that the second page of $\eII_r^{p,q}$ is 
\begin{eqnarray*}
\eII_2^{p,q}=\begin{cases} H^q_m(G;\mathbb{R}),&\text{if } p=0;\\
0, &\text{otherwise.}\end{cases}
\end{eqnarray*}
Hence $\eII_2^{p,q}=\eII_r^{p,q}$ for every $r\geq 2$ and we obtain on the one hand that $\eII_r^{p,q}$ converges to $H_m^*(G;\mathbb{R})$ which is isomorphic to the continuous cohomology $H^*_c(G;\mathbb{R})$ of $G$. On the other hand, we establish
\begin{Pro}
\label{ver}
The spectral sequence $\eI_r^{p,q}$ converges to the cohomology group $H_m^p(G\acts\dH;\mathbb{R})$.

\end{Pro}
\begin{proof}
By definition the first page $\eI_1^{p,q}$ is given by $H_d(K^{p,q})=H^q_m(G;C^p)$. Let $G_\infty=(\mathbb{R}_{>0}\times\mathrm{SO}(n-1))\ltimes\mathbb{R}^{n-1},G_{\infty,0}=\mathbb{R}_{>0}\times\mathrm{SO}(n-1)$ and $G_{\infty,0,1}=\mathrm{SO}(n-2)$ be the stabilizers of respectively $\{\infty\}, \{\infty,0\}$ and $\{\infty,0,1\}$, where these are viewed as points in the upper half space model of $\Hn$. Note that $G/G_\infty\cong\partial\mathbb{H}^n$, and furthermore $G/G_{\infty,0}\hookrightarrow \partial\mathbb{H}^n\times\partial\mathbb{H}^n$ and $G/G_{\infty,0,1}\hookrightarrow \partial\mathbb{H}^n\times\partial\mathbb{H}^n\times\partial\mathbb{H}^n$ as conull sets. By the Eckmann-Shapiro Lemma it then directly follows that 
\begin{align*}
\eI_1^{0,q}=H^q_m(G;C^0) &\cong H_m^q(G_\infty;\mathbb{R}),\\
\eI_1^{1,q}=H^q_m(G;C^1) &\cong H_m^q(G_{\infty,0};\mathbb{R}),\\
\eI_1^{2,q}=H^q_m(G;C^2) &\cong H_m^q(G_{\infty,0,1};\mathbb{R}), \mbox{ and}\\
\eI_1^{p,q}=H^q_m(G;C^p)&\cong H_m^q(G_{\infty,0,1}; C^{p-3}), \mbox{ for } p>2. 
\end{align*}
Since for $\mathbb{R}$-coefficients measurable and continuous cohomology coincide it follows that $\eI_1^{0,q}\cong H_c^q(G_\infty;\mathbb{R})$. By the van Est isomorphism \cite[Chapter IX, Corollary 5.6]{BW} for any connected Lie group $G$ we have that $H^*_c(G;\mathbb{R})\cong H^*(\Omega^*(G/K)^G)$, where $K$ is a maximal compact subgroup of $G$ and where $\Omega^q(G/K)^G$ denotes the set of $G$-invariant real differential $q$-forms on $G/K$. Note that the van Est isomorphism used here is more general than the best known version for semisimple Lie groups since here $G$ is not assumed to be semisimple and therefore the differentials can be non-zero.  
\\
Thus, since a maximal compact subgroup of $G_\infty$ is $\mathrm{SO}(n-1)$, the cohomology group $H^*_c(G_\infty;\mathbb{R})$ can be computed using the complex of multi-linear alternating $\mathrm{SO}(n-1)$-invariant maps $(\mathbb{R}\times\mathbb{R}^{n-1})^q\to \mathbb{R}$, where $\mathrm{SO}(n-1)$ acts on the $\mathbb{R}^{n-1}$ factor. Let $T:\mathbb{R}\times\mathbb{R}^{n-1}\to\mathbb{R}$ be the projection onto the first factor and let $\det:(\mathbb{R}\times\mathbb{R}^{n-1})^{n-1}\to \mathbb{R}$ be the determinant defined on the second factor, i.e.
\begin{align*}
T:(t,v)&\mapsto t,\\
\det:\big((t_1,v_1),\dots,(t_{n-1},v_{n-1})\big)&\mapsto \det(v_1,\dots,v_{n-1}),
\end{align*}
where $t,t_1,\dots,t_{n-1}\in\mathbb{R}$ and $v,v_1,\dots,v_{n-1}\in\mathbb{R}^{n-1}$. 
Then, up to scalar multiplication, the only nonzero alternating forms are constant maps in degree $0$, the form $T$ in degree $1$, the form $\det$ in degree $n-1$ and $T\wedge\det$ in degree $n$. At the end of this section we will prove 

\begin{Le}
\label{differential}
Let $\det\in\Omega^{n-1}(G_\infty/K)^{G_\infty}$ be the form defined above. Then $d(\det)=(1-n)\cdot T\wedge\det$ . 
\end{Le}

It then follows that

\begin{eqnarray*}
H_c^q(G_\infty;\mathbb{R})\cong\begin{cases} \mathbb{R}, &\text{if } q=0,1;\\
0, &\text{otherwise.}\end{cases}
\end{eqnarray*}
Furthermore, a maximal compact of $G_{\infty,0}$ is also $\mathrm{SO}(n-1)$ and thus
\begin{eqnarray*}
\eI_1^{1,q}\cong H^q\left(\Omega^*((\mathbb{R}_{>0}))^{\mathrm{SO}(n-1)}\right)\cong\begin{cases} \mathbb{R}, &\text{if } q=0,1;\\
0, &\text{otherwise,}\end{cases}
\end{eqnarray*}
and, since $\mathrm{SO}(n-2)$ is compact,
\begin{eqnarray*}
\eI_1^{2,q}\cong H_c^q(\mathrm{SO}(n-2);\mathbb{R})\cong\begin{cases}\mathbb{R}, &\text{if q=0;}\\
0, &\text{otherwise.}\end{cases}
\end{eqnarray*}
Hence the first page of $\eI_1^{p,q}$ is as follows:
\begin{figure}[h!]
\includegraphics[scale=0.85]{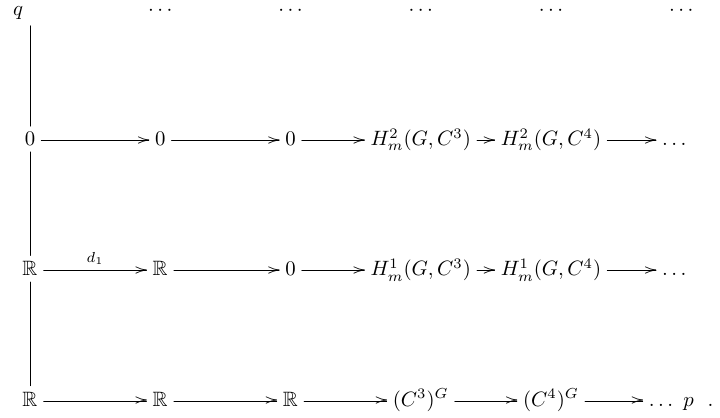}
\label{E_1}
\end{figure}
\newline
We finish the proof of Proposition \ref{ver} using the two following propositions which we will prove in the next two sections. 
\begin{Pro}
\label{d_1}
The map $d_1:\eI_1^{0,1} \to \eI_1^{1,1}$ is an isomorphism.
\end{Pro}
\begin{Pro}
\label{E_2}
$\eI_2^{p,q}=0$ for $p>2$ and $q>0$.
\end{Pro}
It then follows that the spectral sequence degenerates at the second page, that is 
\begin{eqnarray*}
\eI_2^{p,q}= \begin{cases}H_m^p(G\acts\partial\mathbb{H}^n;\mathbb{R}), &\text{if } q=0;\\
0, &\text{otherwise,}\end{cases}
\end{eqnarray*}
which proves Proposition \ref{ver}. 
\end{proof}

Hence $H_c^p(G;\mathbb{R})$ is isomorphic to $H_m^p(G\acts\partial\mathbb{H}^n;\mathbb{R})$ and thus measurably realized on the boundary. This proves the first part of Theorem \ref{iso}. For the second part, note that a locally totally bounded cocycle $(\dH)^{p+1}\to \mathbb{R}$ is essentially bounded. Indeed, since $(\dH)^{p+1}$ is itself compact, a cocycle that sends compact subsets of $(\dH)^{p+1}$ to precompact subsets of $\mathbb{R}$ is bounded a.e. Thus Proposition \ref{bounded} below finishes the proof of Theorem \ref{iso}.
\par
\begin{Pro}
\label{bounded}
Let $A$ be a Polish Abelian $G$-module. Any $G$-invariant cocycle $\alpha:(\partial\mathbb{H}^n)^{p+1}\to A$ is cohomologous to a locally totally bounded cocycle. That is, there exists a $G$-invariant cochain $\sigma:(\partial\mathbb{H}^n)^p\to A$ such that $\kappa=\alpha+\delta\sigma$ is a locally totally bounded cocycle.  
\end{Pro}
\begin{proof}
This follows by a dimension-shifting argument. For $p=0$ cocycles correspond to constants in $A$ and are thus bounded. Suppose now that $p>0$ and define $Q:C((\partial\mathbb{H}^n)^{p+1};A)^G\to C((\partial\mathbb{H}^n)^p;C(G;A))^G$ by
\[
Q\alpha(x_0,\dots,x_{p-1})(g)=(-1)^p\alpha(x_0,\dots,x_{p-1},gG_\infty).
\]
Then
\begin{eqnarray*}
\delta(Q\alpha)(x_0,\dots,x_p)(g)=\alpha(x_0,\dots,x_p),
\end{eqnarray*}
is independent of $g$. Let $\iota(A)$ be the image of the embedding of $A$ into $C(G;A)$ as the closed submodule of constant maps which we identify with $A$. The image $\overline{Q\alpha}$ of $Q\alpha$ in $C((\partial\mathbb{H}^n)^p;C(G;A)/\iota(A))^G$ defines a cocycle. By the induction hypothesis its cohomology class has a representative that is locally totally bounded, i.e.
\[
\overline{Q\alpha}=\overline{\beta}+ \delta\overline{\gamma},
\]
with $\overline{\beta}:(\partial\mathbb{H}^n)^p\to C(G;A)/\iota(A)$ a $G$-invariant locally totally bounded cocycle and $\overline{\gamma}:(\partial\mathbb{H}^n)^{p-1}\to C(G;A)/\iota(A)$ a $G$-invariant cochain. It follows from Lemma \ref{Borel section} that there exists $G$-invariant measurable lifts $\beta:(\dH)^p\to C(G;A)$ and $\gamma:(\dH)^{p-1}\to C(G;A)$ of $\overline{\beta}$ and $\overline{\gamma}$ such that $\beta$ is still locally totally bounded (but no longer a cocycle). Thus
\[
Q\alpha=\beta +\delta\gamma+\sigma,
\]
with $\sigma:(\partial\mathbb{H}^n)^p\to A$ a $G$-invariant measurable cochain. We have $\delta (Q\alpha)=\alpha$ and thus
\[
\alpha=\delta\beta+\delta\sigma.
\]
Hence $\delta\beta$ is a cocycle that takes its values in $\iota(A)\subset C(G;A)$. Therefore $\kappa:=\delta\beta$ can be seen as a cocycle taking its values in $A$ and $\kappa$ is a representative for the cohomology class of $\alpha$. Furthermore, for any compact $L\subset \partial\mathbb{H}^n$:
\[
\kappa(L^{p+1})\subset \beta(L^{p})-\beta(L^{p})+\dots+(-1)^p\beta(L^p),
\]
and thus $\kappa(L^{p+1})$ is precompact as a subset of $A\subset C(G;A)$. It follows that $\kappa$ is a locally totally bounded representative for the cohomology class of $\alpha$. 
\end{proof}

\begin{proof}[Proof of Lemma \ref{differential}]
For $\omega\in\Omega^{n-1}(G_\infty/K)$ and $X_0,\dots,X_{n-1}\in T(G_\infty/K)$ there is the following formula for $d\omega$ (see for example \cite[V.2]{Br})
\begin{align*}
d\omega (X_0,\dots,X_{n-1})=&\sum_{i=0}^{n-1}(-1)^i X_i(\omega(X_0,\dots,\hat{X}_i,\dots,X_{n-1}))\\
&+\sum_{i<j}(-1)^{i+j}\omega([X_i,X_j],X_0,\dots,\hat{X}_i,\dots,\hat{X}_j,\dots,X_{n-1}),
\end{align*}
We identify $G_\infty/K$ with $\mathbb{R}_{>0}\times\mathbb{R}^{n-1}$ by the map
\[
(\lambda, v)\mapsto (\lambda,v)K=\{(\lambda k, kv)\mid k\in K\}
\]
So in particular $K$ is identified with $(1,0)\in\mathbb{R}_{>0}\times\mathbb{R}^{n-1}$. 
Let $X_0,\dots,X_{n-1}\in T(G_\infty/K)$ be the constant vector fields defined by $(X_0)_p=(1,0)$ and $(X_i)_p=(0,e_i)$ for all $p\in G_\infty/K$ and $i=1,\dots, n-1$. Then

\begin{eqnarray*}
&&d(\det)_{K}((X_0)_{K},(X_1)_{K},\dots,(X_{n-1})_{K})\\
&&d(\det)_{(1,0)}\left((1,0),(0,e_1),\dots,(0,e_{n-1})\right)\\
&&=\frac{d}{dt}\bigg |_{t=0}\text{det}_{(1,0)+t(1,0)}\left((0,e_1),\dots,(e_{n-1},0)\right)\\
&&+\sum_{i=1}^{n-1}\frac{d}{dt}\bigg|_{t=0}\text{det}_{(1,0)+t(0,e_i)}\left((1,0),(0,e_1)\dots,(\widehat{0,e_i}),\dots,(0,e_{n-1})\right)\\
&&=\frac{d}{dt}\bigg|_{t=0}\text{det}_{(1+t,0)}\left((0,e_1),\dots,(0,e_{n-1})\right),
\end{eqnarray*}
where the last equality follows from $\det\left((1,0),\cdot,\dots,\cdot\right)=0$
\\
Since $\det$ is $G_\infty$-invariant, i.e. invariant under taking the pullback by the action of $G_\infty$ on $G_\infty/K$, we have 
\[
\text{det}_{(1+t,0)}\left((0,e_1),\dots,(0,e_{n-1}))=\text{det}_{g\cdot (1+t,0)}(g\cdot(0,e_1),\dots,g\cdot (0,e_{n-1})\right),
\] 
for all $g\in G_\infty$. 
Let $g=(\frac{1}{1+t}I_{n-1},0)\in G_\infty$. Then 
\[
g\cdot (1+t,0)=(\frac{1}{1+t}I_{n-1}\cdot(1+ t),0)=(1,0),
\] 
and
\[
g\cdot (0,e_i)=(0,\frac{1}{1+t}I_{n-1}\cdot e_i)=(0,\frac{1}{1+t}e_i).
\]
Hence
\begin{align*}
\text{det}_{(1+t,0)}\left((0,e_1),\dots,(0,e_{n-1})\right)=&\text{det}_{g\cdot (1+t,0)}\left(g\cdot(0,e_1),\dots,g\cdot (0,e_{n-1})\right)\\
=&\text{det}_{(1,0)}\left((0,\frac{1}{1+t}e_1),\dots,(0,\frac{1}{1+t}e_{n-1})\right)\\
=&\det\left(\frac{1}{1+t}e_1,\dots,\frac{1}{1+t}e_{n-1}\right)\\
=&\frac{1}{(1+t)^{n-1}}\det(e_1,\dots,e_{n-1})\\
=&\frac{1}{(1+t)^{n-1}}
\end{align*}
It follows that 
\begin{eqnarray*}
\frac{d}{dt}\bigg|_{t=0}\text{det}_{0+t(1,0)}\left((0,e_1),\dots,(0,e_{n-1})\right)&=&\frac{d}{dt}\bigg |_{t=0}\frac{1}{(1+t)^{n-1}}=1-n
\end{eqnarray*}
Since furthermore $T\wedge \det \left((1,0),(0,e_1),\dots,(0,e_{n-1})\right)=1$ the result follows. 
\end{proof}

\subsection{Proof of Proposition \ref{d_1}}
Recall that $G_\infty=(\mathbb{R}_{>0}\times\mathrm{SO}(n-1))\ltimes\mathbb{R}^{n-1}$ and $G_{\infty,0}=\mathbb{R}_{>0}\times\mathrm{SO}(n-1)$. Let 
\begin{equation*}
j^*:H^1_m(G;C(G/G_\infty;\mathbb{R}))\to H^1_m(G;C(G/G_{\infty,0};\mathbb{R}))
\end{equation*}
be the map induced by the natural surjection $G/G_{0,\infty}\twoheadrightarrow G/G_\infty$, i.e. on cochains
\begin{equation}
\label{j^*}
j^*(\alpha)(g_0,g_1)(gG_{\infty,0})=\alpha(g_0,g_1)(gG_\infty),
\end{equation}
for $\alpha\in C(G^2;C(G/G_{\infty};\mathbb{R}))^G$ and $g_0,g_1,g\in G$. By abuse of notation we will denote the map from $H^1_m(G;C(G/G_\infty;\mathbb{R}))$ to $H^1_m(G;C(G/G_{\infty,0};\mathbb{R}))$ that is induced by the differential $d_1:H^1_m(G;C^0)\to H^1_m(G;C^1)$ also by $d_1$. We will prove
\begin{Pro}
The map 
\[
d_1:H^1_m(G;C(G/G_\infty;\mathbb{R}))\to H^1_m(G;C(G/G_{\infty,0};\mathbb{R}))
\]
is equal to $2j^*$.
\end{Pro}
This proves Proposition \ref{d_1} since it implies in particular that $d_1$ is an isomorphism. Indeed, measurable cohomology and continuous cohomology agree in degree $1$ and thus
\[
H^1_m(G_\infty;\mathbb{R})=\text{Hom}(G_\infty;\mathbb{R})\cong\mathbb{R}
\]
and a generator is given by the homomorphism $f_1:G_\infty\to\mathbb{R}$ defined by
\[
f_1(kA+v)=\ln(k),
\]
where we write $kA+v$  for the element $((k,A),v)\in G_\infty=(\mathbb{R}_{>0}\times \mathrm{SO}(n-1))\ltimes \mathbb{R}^{n-1}$. Furthermore,
\[
H^1_m(G_{\infty,0};\mathbb{R})=\text{Hom}(G_{\infty,0};\mathbb{R})\cong\mathbb{R},
\]
and a generator is given by the homomorphism $f_2:G_{\infty,0}\to\mathbb{R}$ defined by
\begin{equation}
\label{f2}
f_2(kA)=\ln(k),
\end{equation}
for $kA\in G_{\infty,0}=\mathbb{R}_{>0}\times \mathrm{SO}(n-1)$. Under the Eckmann-Shapiro Lemma $j^*$ corresponds to the map $i^*:H^1_m(G_\infty;\mathbb{R})\to H^1_m(G_{\infty,0};\mathbb{R})$ induced by the natural inclusion $i:G_{\infty,0}\hookrightarrow G_\infty$. This is an isomorphism as it sends $f_1$ to $f_2$. 
Let $J\in\text{Isom}^+(\mathbb{H}^n)$ be a rotation by $\pi$ centered on a point on the geodesic between $0$ and $\infty$ so that $J(0)=\infty$, $J(\infty)=0$ and $J^{-1}=J$. In the upper half space model a possible formula for $J$ is
\begin{equation}
\label{J}
J:(x_1,\dots,x_n)\mapsto \frac{1}{|x|^2}(x_1,\dots,x_{n-2},-x_{n-1},x_n),
\end{equation}
 for $x=(x_1,\dots,x_n)\in \overline{\mathbb{H}^n}\setminus \{\infty\}$. Let
\[
 J^*:H^1_m(G;C(G/G_{\infty,0};\mathbb{R}))\to H^1_m(G;C(G/G_{\infty,0};\mathbb{R}))
\] 
be the isomorphism defined on cochains by
\begin{equation}
\label{J^*}
J^*(\alpha)(g_0,g_1)(gG_{\infty,0})=\alpha(g_0,g_1)(gJG_{\infty,0}),
\end{equation}
for $\alpha\in C(G^2;C(G/G_{\infty,0};\mathbb{R}))^G$ and $g_0,g_1,g\in G$. 
Let 
\[
\psi_\infty:C(G^2;C(G/G_\infty;\mathbb{R}))^G\to C(G^2;C(\partial\mathbb{H}^n;\mathbb{R}))^G
\]
be the isomorphism defined by  
\begin{equation*}
\label{psi.infty}
\psi_\infty(\beta)(g_0,g_1)(x)=\beta(g_0,g_1)(gG_\infty), 
\end{equation*}
for $\beta\in C(G^2;C(G/G_\infty;\mathbb{R}))^G$, $g_0,g_1\in G$ and $x\in\partial\mathbb{H}^n$ and with $g\in G$ such that $g\cdot \infty=x$. 
Furthermore, let 
\[
\psi_{\infty,0}:C(G^2;C(\partial\mathbb{H}^n\times\partial\mathbb{H}^n;\mathbb{R}))^G\to C(G^2;C(G/G_{\infty,0};\mathbb{R}))^G
\]
be the isomorphism defined by
\begin{equation*}
\label{psi.infty,0}
\psi_{\infty,0}(\alpha)(g_0,g_1)(gG_{\infty,0})=\alpha(g_0,g_1)(g\cdot 0,g\cdot \infty),
\end{equation*}
for $\alpha\in C(G^2;C(\partial\mathbb{H}^n\times\partial\mathbb{H}^n;\mathbb{R}))^G$ and $g_0,g_1,g\in G$. We have the commutative diagram 
\begin{figure}[h!]
\center
\begin{tikzcd}
C(G^2; C(G/G_\infty;\mathbb{R}))^G\arrow{r}{d_1}\arrow{d}{\psi_\infty}
&C(G^2;C(G/G_{\infty,0};\mathbb{R}))^G\\
C(G^2;C^0)^G\arrow{r}{\delta}
&C(G^2;C^1)^G,\arrow{u}{\psi_{\infty,0}}
\end{tikzcd}
\label{diagramd_1}
\end{figure}
\newline
and furthermore, with $j^*$ and $J^*$ the maps defined in equation \eqref{j^*} and equation \eqref{J^*} respectively, we obtain 
\begin{Le}
$d_1=j^*-J^*\circ j^*$. 
\end{Le}
\begin{proof}
By definition the map $d_1:H^1_m(G;C^0)\to H^1_m(G;C^1)$ is induced by  $\delta:C^0\to C^1$, i.e. for $[\sigma]\in H^1_m(G;C^0), g\in G$ and $x_0,x_1\in\partial \mathbb{H}^{n}$
\[
d_1[\sigma]=[\delta\circ\sigma], 
\]
where
\[
(\delta\circ\sigma)(g_0,g_1)(x_0,x_1)=\sigma(g_0,g_1)(x_1)-\sigma(g_0,g_1)(x_0). 
\]
For $\sigma\in C(G^2;C(G/G_\infty;\mathbb{R}_{>0}))$ and $g_0,g_1,g\in G$ we have 
\begin{eqnarray*}
\psi_{\infty,0}\circ\delta\circ\psi_\infty(\sigma)(g_0,g_1)(gG_{\infty,0})&=& \delta\circ\psi_\infty(\sigma)(g_0,g_1)(g\cdot 0, g\cdot\infty)\\
&=&\psi_\infty(\sigma)(g_0,g_1)(g\cdot\infty)\\
&&-\psi_\infty(g_0,g_1)(g\cdot 0)\\
&=&\sigma(g_0,g_1)(gG_\infty)-\sigma(g_0,g_1)(g\cdot J G_\infty)\\
&=&j^*(\sigma)(g_0,g_1)(gG_{\infty,0})\\
&&- j^*(\sigma)(g_0,g_1)(g\cdot J G_{\infty,0})\\
&=&j^*(\sigma)(g_0,g_1)(gG_{\infty,0})\\
&&-J^*\circ j^*(\sigma)(g_0,g_1)(gG_{\infty,0}),
\end{eqnarray*}
and it thus follows that $d_1=j^*-J^*\circ j^*$. 
\end{proof}
Since continuous cohomology and measurable cohomology coincide in degree $1$ we can and will from now on work with continuous cochains. For such cochains an isomorphism $\varphi:C_c(G_{\infty,0}^2;\mathbb{R})^{G_{\infty,0}}\to C_c(G_{\infty,0};\mathbb{R})$ between the homogeneous and inhomogeneous resolution is given by
\begin{eqnarray*}
\varphi(\beta)(g)&=&\beta(e,g)  \text{, with inverse}\nonumber\\
\label{varphi}
\varphi(\sigma)^{-1}(g_0,g_1)&=&g_0\cdot\sigma(g_0^{-1}g_1).
\end{eqnarray*}
\begin{Le}
$J^*$ acts as $-1$ on $H^1_m(G_{\infty,0};\mathbb{R})$.
\end{Le}
\begin{proof}
Let $s:G/G_{\infty,0}\to G$ be a Borel section such that $s(G_{\infty,0})=e$. Let $\alpha\in \text{Hom}(G_{\infty,0};\mathbb{R}_{>0})$ be a cocycle, $h_1\in G_{\infty,0}$, and $g\in G$. Let $\phi$ and $\psi$ be the maps defined in equation \eqref{phi} and equation \eqref{psi} respectively. Then
\begin{eqnarray*}
&&\varphi\circ \phi\circ J^*\circ\psi\circ\varphi^{-1}(\alpha)(h_1)\\
&&=\phi\circ J^*\circ\psi\circ\varphi^{-1}(\alpha)(e,h_1)\\
&&= J^*\circ\psi\circ\varphi^{-1}(\alpha)\big(s(gG_{\infty,0}),s(gG_{\infty,0})h_1\big)(gG_{\infty,0})\\
&&=\psi\circ\varphi^{-1}(\alpha)(s(gG_{\infty,0}),s(gG_{\infty,0})h_1)(gJG_{\infty,0})\\
&&=\varphi^{-1}(\lambda(s(gG_{\infty,0}),gJG_{\infty,0}),\lambda(s(gG_{\infty,0})h_1, gJG_{\infty,0}))\\
&&=\alpha(\lambda(s(gG_{\infty,0}), gJG_{\infty,0})^{-1}\cdot\lambda(s(gG_{\infty,0})h_1,gJG_{\infty,0}))\\
%&& =\alpha((s(gJG_{\infty,0})^{-1}\cdot s(gG_{\infty,0})\cdot s(s(gG_{\infty,0})^{-1}gJG_{\infty,0}))^{-1} \cdot s(gJG_{\infty,0})^{-1}\cdot s(gG_{\infty,0})h_1\cdot s( (s(gG_{\infty,0})h_1)^{-1}gJG_{\infty,0}))\\
&&=\alpha\big(s\big(s(gG_{\infty,0})^{-1}gJG_{\infty,0}\big)^{-1}\cdot h_1\cdot s\big((s(gG_{\infty,0})h_1)^{-1}gJG_{\infty,0}\big)\big)\\
&&=\alpha\big(s(h_2^{-1}g^{-1}gJG_{\infty,0})^{-1}\cdot h_1\cdot s(h_1^{-1}h_2^{-1}g^{-1}gJG_{\infty,0})\big)\\
&&=\alpha\big(s(h_2^{-1}JG_{\infty,0})^{-1}\cdot h_1\cdot s(h_1^{-1}h_2^{-1}JG_{\infty,0})\big)\\
&&=\alpha(s(JG_{\infty,0})^{-1}\cdot h_1\cdot s(JG_{\infty,0}))\\
&&=\alpha(h_3^{-1}J^{-1}h_1Jh_3)\\
&&=\alpha(Jh_1J),
\end{eqnarray*}
where $h_2=g^{-1}s(gG_{\infty,0})\in G_{\infty,0}$, $h_3= Js(JG_{\infty,0})\in G_{\infty,0}$ and we use that $s(hJG_{\infty,0})=s(hG_{\infty,0})\cdot(JG_{\infty,0})=s(JG_{\infty,0})$ for all $h\in G_{\infty,0}$. Thus $J^*$ acts by conjugation on $H^1_m(G_{\infty,0};\mathbb{R}_{>0})$. 
\\\\
Let $g\in G_{\infty,0}$, say $g=kA$ with $k>0$ and $A\in \mathrm{SO}(n-1)$ and let $x\in\overline{\mathbb{H}^n}\setminus \{\infty\}$. Recall the formula for $J$ given in equation \eqref{J} and note that $J(x)=\frac{1}{|x|^2}r(x)$ with 
\[
r:(x_1,\dots,x_{n-2},x_{n-1},x_n)\mapsto (x_1,\dots,x_{n-2},-x_{n-1},x_n)
\]
the reflection in the hyperplane orthogonal to the $(n-1)$th coordinate axis. Then
\begin{eqnarray*}
JgJ(x)&=&J\cdot \frac{k}{|x|^2}A \cdot r(x)\\
&=& \frac{1}{\left| \frac{k}{|x|^2}\cdot |A\cdot r(x)|\right|^2}\cdot  \frac{k}{|x|^2}rAr(x)\\
&=&  \frac{1}{\left| \frac{k}{|x|^2}\cdot |x|\right|^2}\cdot  \frac{k}{|x|^2}rAr(x)\\
&=&\frac{1}{k} rAr(x),
\end{eqnarray*}
Hence for the generator $f_2$ defined above in equation \eqref{f2} we obtain that
\begin{align*}
f_2(JgJ)&=f_2\left(\frac{1}{k}rAr\right)=\ln\left(\frac{1}{k}\right)=-\ln(k) =-f_2(kA) \\
&=-f_2(g).
\end{align*}

\end{proof}
\subsection{Vanishing of $\eI_2^{p,q}$ for $p>2$ and $q>0$}
In this section we give the proof of Proposition \ref{E_2}. Recall that $G=\text{Isom}^+(\Hn)$ and that $K^{p,q}=C(G^{q+1};C((\partial\mathbb{H}^n)^{p+1};\mathbb{R}))^G$ with differentials $\delta:K^{p,q}\to K^{p+1,q}$ and $d:K^{p,q}\to K^{p,q+1}$ given by the standard homogeneous coboundary operators. For a Polish Abelian $G$-module $A$ denote by $K^{p,q}(A)$ the $G$-module $C(G^{q+1};C((\partial\mathbb{H}^n)^{p+1};A))^G$. We will identify it with the $G$-module of $G$-invariant measurable maps $G^{q+1}\times(\partial\mathbb{H}^n)^{p+1}\to A$. We write $[[\alpha]_d]_\delta\in H_\delta H_d(K^{p,q}(A))$ if  a cochain $\alpha\in C(G^{q+1};C((\partial\mathbb{H}^n)^{p+1};A))^G$ satisfies $d\alpha=0$ and $\delta\alpha=d\gamma$, where $\gamma:G^q\times(\partial\mathbb{H}^n)^{p+2}\to A$. 
\begin{Pro}
\label{ltb2}
For all $p,q\in\mathbb{N}$, all Polish Abelian $G$-modules $A$ and all $[[\alpha]_d]_\delta\in H_\delta H_d(K^{p,q}(A))$ there exists a locally totally bounded representative $\kappa$ of $[[\alpha]_d]_\delta$. That is, there exist $G$-invariant measurable maps $\sigma:G^{q+1}\times(\partial\mathbb{H}^n)^{p}\to A$ and $\lambda:G^q\times (\partial\mathbb{H}^n)^{p+1}\to A$ such that 
\[
\kappa=\alpha+\delta\sigma+d\lambda:G^{q+1}\times (\partial\mathbb{H}^n)^{p+1}\to A
\]
is locally totally bounded. Furthermore, $\sigma$  can be chosen such that $d\sigma=0$.
\end{Pro}
\begin{proof}
We will prove Proposition \ref{ltb2} by induction on $q$ while allowing the module $A$ to vary, thereby proving it for all Abelian $G$-modules $A$. Suppose $q=0$. Then $d\alpha=0$ implies that the cocycle $\alpha:G\to C((\partial\mathbb{H}^n)^{p+1};\mathbb{R})$ is a constant function into $C((\partial\mathbb{H}^n)^{p+1};A)^G$. We identify $\alpha$ with this element of $C((\partial\mathbb{H}^n)^{p+1};A)^G$ and it thus follows from Proposition \ref{bounded} that Proposition \ref{ltb2} holds in degree $q=0$ for all Polish Abelian $G$-modules $A$. \\
Suppose now that $q>0$ and that for $q'<q$ the proposition is true for all Polish Abelian $G$-modules $A$. Let $\alpha\in C(G^{q+1};C((\partial\mathbb{H}^n)^{p+1};A))^G$ be such that $d\alpha=0$ and $\delta\alpha=d\gamma$, where $\gamma:G^q\times(\partial\mathbb{H}^n)^{p+2}\to A$. We can assume that $\delta\alpha=0$. Indeed, if $\delta\alpha=d\gamma\neq 0$ then for any $x\in\partial\mathbb{H}^n$ such that 
\[
\delta^2\alpha(g_0,\dots,g_q)(x_0,\dots,x_{p+1},x)=0
\]
for almost every $((g_0,\dots,g_q),(x_0,\dots,x_{p+1}))\in G^{q+1}\times(\partial\mathbb{H}^n)^{p+2}$ the cocycle
\begin{align*}
\alpha_x(g_0,\dots,g_q)(x_0,\dots,x_p):=&\alpha(g_0,\dots,g_q)(x_0,\dots,x_p)\\
										&+(-1)^p \delta\alpha(g_0,\dots,g_q)(x_0,\dots,x_p,x)
\end{align*}
is a representative of the class of $\alpha$ (since $\delta\alpha=d\gamma$ is a coboundary) and
\begin{align*}
\delta\alpha_x(g_0,\dots,g_q)(x_0,\dots,&x_{p+1})\\
										:=&\sum_{i=0}^{p+1}(-1)^i[\alpha(g_0,\dots,g_q)(x_0,\dots,\hat{x}_i,\dots,x_{p+1})\\
										&+(-1)^p \delta\alpha(g_0,\dots,g_q)(x_0,\dots,\hat{x}_i,\dots,x_{p+1},x)]\\
										=& \delta\alpha(g_0,\dots,g_q)(x_0,\dots,x_{p+1})\\
										&+(-1)^p\delta^2\alpha(g_0,\dots,g_q)(x_0,\dots,x_{p+1},x)\\
										&-(-1)^{p}\cdot(-1)^{p+2}\delta\alpha(g_0,\dots,g_q)(x_0,\dots,x_{p+1})\\
										=& 0.
\end{align*}

Now define the function $Q\alpha$ in $C(G^q;C((\partial\mathbb{H}^n)^{p+1};C(G;A)))^G$ by
\[
Q\alpha(g_0,\dots,g_{q-1})(x_0,\dots, x_p)(g):= (-1)^{q+1}\alpha(g_0,\dots,g_{q-1},g)(x_0,\dots,x_p).
\]
Then
\begin{align*}
d(Q\alpha)(g_0,\dots,g_q)(x_0,\dots,&x_p)(g)\\
=&\sum_{i=0}^q (-1)^i Q\alpha (g_0,\dots,\hat{g}_i,\dots, g_q)(x_0,\dots,x_p)(g)\\
=&\sum_{i=0}^q (-1)^{i+q+1}\alpha(g_0,\dots,\hat{g}_i,\dots,g_q,g)(x_0,\dots,x_p)\\
=&(-1)^{q+1} d\alpha(g_0,\dots,g_q,g)(x_0,\dots,x_p)\\
&-(-1)^{2q+1}\alpha(g_0,\dots,g_q)(x_0,\dots,x_p)\\
=&\alpha(g_0,\dots,g_q)(x_0,\dots,x_p),
\end{align*}
and we see that $d(Q\alpha)$ takes its values in $\iota(A)$ and therefore the image $\overline{Q\alpha}$ of $Q\alpha$ in  $C(G^q;C((\partial\mathbb{H}^n)^{p+1};C(G;A)/\iota(A)))^G$ is a cocycle with respect to the coboundary operator $d$. Furthermore,
\begin{align*}
\delta (Q\alpha)(g_0,\dots,g_{q-1})(x_0,\dots,&x_{p+1})(g)\\
 =&(-1)^{q+1} \delta\alpha(g_0,\dots,g_{q-1},g)(x_0,\dots,x_{p+1})\\
=&0
\end{align*}
Hence $\overline{Q\alpha}\in C(G^q;C((\partial\mathbb{H}^n)^{p+1};C(G;A)/\iota(A)))^G$ represents a cohomology class in $H_\delta H_d(K^{p,q-1}(C(G;A)/\iota(A)))$ and so by the induction hypothesis
\[
\overline{Q\alpha}=\bar{\beta}+\delta\bar{\mu}+d\bar{\nu},
\]
where $\bar{\beta}$ is a locally totally bounded cocycle and $d\bar{\mu}=0$. Then, by Lemma \ref{Borel section}, there exist $G$-invariant measurable lifts $\beta,\mu$ and $\nu$ of these maps such that
\[
Q\alpha=\beta +\delta\mu+d\nu +\eta,
\]
with $\eta:G^q\times (\partial\mathbb{H}^n)^{p+1}\to \iota(A)$ and such that $\beta$ is still locally totally bounded. We obtain
\[
\alpha=d(Q\alpha)=d\beta+d\delta\mu +d\eta=d\beta+\delta(d\mu)+d\eta.
\]
Note that since $\alpha$ takes values in $\iota(A)$ the right-hand side does as well. Also, since $d\bar{\mu}=0$, $d\mu$ takes values in $\iota(A)$ and hence $d\delta\mu=\delta(d\mu)$ can be identified with a coboundary in $C(G^{q+1};C((\partial\mathbb{H}^n)^{p+1};A))^G$. The map $\eta$ also takes its values in $\iota(A)$ and it thus follows that $\kappa:=d\beta$ is a locally totally bounded representative of the class of $\alpha$ in $H_\delta H_d(K^{p,q}(A))$.

\end{proof}
Let $K_c^{p,q}(A)$ be the $G$-module $C_c(G^{q+1};C((\partial\mathbb{H}^n)^{p+1};A))^G$ where $A$ is from now on a Fr\'echet $G$-module and let $K_c^{p,q}=K_c^{p,q}(\mathbb{R})$. Then
\begin{Pro}
\label{cocycles continuous}
$H_\delta H_d(K_c^{p,q}(A))\cong H_\delta H_d(K^{p,q}(A))$.
\end{Pro}
\begin{proof}
By Proposition \ref{ltb2} any cohomology class $[[\kappa]_d]_\delta\in H_\delta H_d({K}^{p,q}(A))$ has a locally totally bounded representative $\kappa$. Then, as in the proof of T. Austin and C.C. Moore of the isomorphism between continuous and measurable cohomology for Fr\'echet coefficients, such a locally totally bounded cocycle is effaced by the inclusion $A\hookrightarrow C_c(G;A)$ and the result will follow by Buchsbaum's criterion. More precisely, let us check that the three conditions of Buchsbaum's criterion (Lemma \ref{BuchsbaumLemma}) hold for both sides of the isomorphism. Firstly, 
\[
H_\delta H_d(K^{p,0}(A))=H_\delta H_d(K_c^{p,0}(A))=H_\delta((C((\partial\mathbb{H}^n)^{p+1};A))^G),
\]
so both sides agree in degree zero. Furthermore, the existence of continuous cross sections for Fr\'echet modules ensures that condition \eqref{cond2} holds for both sides. 
For condition \eqref{cond3} we claim that the short sequence 
\[
\xymatrix{
0\ar[r] & A \ar[r]^-\iota & C_c(G;A)\ar[r] & C_c(G;A)/\iota(A)\ar[r] & A
}
\]

effaces all $[[\kappa]_d]_\delta\in H_\delta H_d({K_{(c)}}^{p,q}(A))$.
Indeed, assume that $\kappa$ is a locally totally bounded representative of the class $[[\kappa]_d]_\delta$ (which is automatic if $[[\kappa]_d]_\delta\in H_\delta H_d({K_{c}}^{p,q}(A)))$. Then there exists an $\eta:G^q\times (\partial\mathbb{H}^n)^{p+1}\to C_c(G;A)$ s.t. $d\eta=\kappa$ where $\kappa$ is viewed as a map $G^{q+1}\times(\partial\mathbb{H}^n)^{p+1}\to C_c(G;A)$ taking values in $\iota(A)\subset C_c(G;A)$. For example, we can define $\eta$ by
\begin{eqnarray*}
&&\eta(g_0,\dots,g_{q-1})(x_0,\dots,x_p)(g):=\\
&&(-1)^q\int_G\kappa(g_0,\dots,g_{q-1},gh)(x_0,\dots,x_p)\xi(h)d\mu_G(h),
\end{eqnarray*}
where $g_0,\dots,g_{q-1},g\in G$, $x_0,\dots,x_p\in \partial\mathbb{H}^n$ and $\xi:G\to\mathbb{R}_{>0}$ is a compactly-supported continuous function with $\int_G \xi d\mu_G=1$. Then
\begin{align*}
d\eta&(g_0,\dots,g_{q})(x_0,\dots,x_p)(g)\\
&=\sum_{i=0}^q (-1)^i(-1)^q\int_G \kappa(g_0,\dots,\hat{g}_i,\dots,g_q,gh)(x_0,\dots,x_p)\xi(h)d\mu_G(h)\\
&=(-1)^q \int_G d\kappa(g_0,\dots,g_q,gh)(x_0,\dots,x_p)\xi(h)d\mu_G(h)\\
&\hspace{.5cm} -(-1)^q(-1)^{q+1} \int_G \kappa(g_0,\dots,g_q)(x_0,\dots,x_p)\xi(h) d\mu_G(h)\\
&=\kappa(g_0,\dots,g_q)(x_0,\dots,x_p),
\end{align*}
where the last equality follows from the fact that $\kappa$ is a cocycle. It follows that the image of $[[\kappa]_d]_\delta$ in $ H_\delta H_d({K_{c}}^{p,q}(C_c(G;A)))$ is zero.

\end{proof}
Let $K=\text{SO}(n-2)$ and let $s:G/K\to G$ be a locally totally bounded Borel section such that $s(K)=e$. By the Eckmann-Shapiro Lemma we have the isomorphism $H_m^q(G;C(G/K;C^{p-3}))\cong H_m^q(K;C^{p-3})$. From Proposition \ref{cocycles continuous} it follows that if we restrict to cocycles in $ H_\delta H_d(K^{p,q})$ we can assume them to be continuous in $G^{q+1}$. As shown in section \ref{ES Lemma} we then get an explicit map $\phi:C_c(G^{q+1}; C(G/K;C^{p-3}))^G\to C_c(K^{q+1};C^{p-3})^K$ that induces the isomorphism 
\[
H_\delta H_d(C(G^{q+1}; C(G/K;C^{p-3}))^G)\cong H_\delta H_d (C(K^{q+1};C^{p-3})^K).
\]
\begin{proof}[Proof of Proposition \ref{E_2}]
Let $[[\alpha]_d]_\delta\in H_\delta H_d(K_c^{p,q})=E_2^{p,q}$. We will show that in this case $\alpha$ is cohomologous in $H_\delta H_d (K_c^{p,q})$ to a coboundary in $H_d(K^{p,q})$. By Proposition \ref{ltb2}, the cohomology class of $\alpha$ has a locally totally bounded representative 
\[
\beta:G^{q+1}\times (\partial\mathbb{H}^n)^{p+1}\to \mathbb{R}.
\]
Then $\phi(\beta):K^{q+1}\times (\partial\mathbb{H}^n)^{p-2}\to \mathbb{R}$  is also a locally totally bounded cocyle and furthermore we have $\phi(\beta)=d\eta$, where $\eta:K^{q}\times (\partial\mathbb{H}^n)^{p-2}\to \mathbb{R}$ is defined by  
\begin{eqnarray*}
&&\eta(k_0,\dots,k_{q-1})(x_0,\dots,x_{p-3}):=\\
&& (-1)^q\int_K \phi(\beta)(k_0,\dots,k_{q-1},k)(x_0,\dots,x_{p-3}) d\mu_K(k).
\end{eqnarray*}
It follows that $\eI_2^{p,q}=0$ for $p>2$ and $q>0$. 
\end{proof}
\section{Injectivity of the comparison map in degree 3}
Corollary \ref{inj}, i.e. injectivity of the comparison map for real hyperbolic space $\mathbb{H}^n$ in degree 3, is an immediate consequence of Theorem \ref{iso}. By this theorem, we have 
\[
H_c^3(G;\mathbb{R})=\frac{\ker(\delta:C((\partial\mathbb{H}^n)^4;\mathbb{R})^G\to C((\partial\mathbb{H}^n)^5;\mathbb{R})^G)}{\text{im}(\delta:C((\partial\mathbb{H}^n)^3;\mathbb{R})^G\to C((\partial\mathbb{H}^n)^4;\mathbb{R})^G)}.
\]
Furthermore, the continuous bounded cohomology of $G$ can also be computed with maps that are defined on the boundary of hyperbolic space \cite[Theorem 7.5.3]{Mon}. That is,
\[
H^3_{c,b}(G;\mathbb{R})=\frac{\ker(\delta:L^\infty((\partial\mathbb{H}^n)^4;\mathbb{R})^G\to L^\infty((\partial\mathbb{H}^n)^5;\mathbb{R})^G)}{\text{im}(\delta:L^\infty((\partial\mathbb{H}^n)^3;\mathbb{R})^G\to L^\infty((\partial\mathbb{H}^n)^4;\mathbb{R})^G)},
\]
where $L^\infty((\partial\mathbb{H}^n)^p;\mathbb{R})\subset C((\partial\mathbb{H}^n)^p;\mathbb{R})$ consists of essentially bounded measurable function classes. By $3$-transitivity of the action of $G$ on the boundary of hyperbolic space it then follows that cochains in degree $2$ are constant. Since in even degree applying $\delta$ to a constant gives zero there are no coboundaries in degree $3$. Hence $H_c^3(G;\mathbb{R})$ and $H^3_{c,b}(G;\mathbb{R})$ are equal to the corresponding spaces of cocycles and it follows that the comparison map is injective, given as the natural inclusion of cocycles. 
\par
Injectivity in degree $3$ for Isom$^+(\mathbb{H}^n)$ also follows from a simpler argument which only uses some basic properties of hyperbolic space and the injectivity in degree 3 for $n=3$. Denote by $\mathbb{R}_\epsilon$ the Isom$(\mathbb{H}^n)$-module $\mathbb{R}$ with Isom$(\mathbb{H}^n)$-action given by the homomorphism
 \[
 \epsilon :\text{Isom}(\mathbb{H}^n)\to\text{Isom}(\mathbb{H}^n)/\text{Isom}^+(\mathbb{H}^n)\cong\{1,-1\}.
 \] 
We have
\[
H^*_{c,(b)}(\text{Isom}^+(\mathbb{H}^n);\mathbb{R})\cong H^*_{c,(b)}(\text{Isom}(\mathbb{H}^n);\mathbb{R})\oplus H^*_{c,(b)}(\text{Isom}(\mathbb{H}^n);\mathbb{R}_\epsilon).
\]
Also, $H^3_{c,(b)}(\text{Isom}^+(\mathbb{H}^3);\mathbb{R})$ is generated by the volume cocycle which is equivariant, i.e. in $H^3_{c,(b)}(\text{Isom}(\mathbb{H}^3);\mathbb{R}_\epsilon)$, and thus $H^3_{c,(b)}(\text{Isom}(\mathbb{H}^3);\mathbb{R})=0$. On the other hand, for $n>3$ we have $H^3_{c,(b)}(\text{Isom}(\mathbb{H}^n);\mathbb{R}_\epsilon)=0$ (see \cite{BBI}). Thus it follows that $H^3_{c,(b)}(\text{Isom}^+(\mathbb{H}^n);\mathbb{R})=H^3_{c,(b)}(\text{Isom}(\mathbb{H}^n);\mathbb{R})$. 
\begin{Le}
\label{nul}
Let $n>3$. Then $H^3_{c,b}(\mathrm{Isom}(\mathbb{H}^{n});\mathbb{R})=0$. 
\end{Le}
\begin{proof}
Let $[\beta]\in H_{c,b}^3(\text{Isom}(\mathbb{H}^{n});\mathbb{R})$. Let $i:\mathbb{H}^3\hookrightarrow \mathbb{H}^{n}$ be a natural embedding. We will identify the image $i(\mathbb{H}^3)\subset\mathbb{H}^{n}$ with $\mathbb{H}^3$. Since $\beta$ restricted to $\mathbb{H}^3$ is a cocycle in $C_{c,b}((\mathbb{H})^3;\mathbb{R})^{\text{Isom}(\mathbb{H}^3)}$ there is an $\alpha\in C_{c,b} ((\mathbb{H}^3)^3;\mathbb{R})^{\text{Isom}(\mathbb{H}^3)}$ such that $\delta\alpha=\beta |_{\mathbb{H}^3}$. Let $x_0,x_1,x_{2}\in\mathbb{H}^{n}$. These points lie in the $2$-dimensional linear plane $H(x_0,x_1,x_{2})$ spanned by the points $x_0,x_1,x_2$. Since $\text{Isom}(\mathbb{H}^{n})$ acts transitively on linear $2$-planes there always exists a $g\in \text{Isom}(\mathbb{H}^{n})$ such that $g (H(x_0,x_1,x_{2}))\subset \mathbb{H}^3$. Define $\bar{\alpha}:(\mathbb{H}^{n})^3\to\mathbb{R}$ by 
\[
\bar{\alpha}(x_0,x_1,x_{2})=\alpha(gx_0,gx_1,gx_{2}),
\]
 where $g\in\text{Isom}(\mathbb{H}^{n})$ is such that $g(H(x_0,x_1, x_{2}))\subset \mathbb{H}^3$. One checks easily that this is well defined,  $\bar{\alpha}\in C_{c,b}((\mathbb{H}^{n})^{3};\mathbb{R})^{\text{Isom}(\mathbb{H}^{n})}$ and $\delta \bar{\alpha}=\beta$. Hence it follows that $H^3_{c,b}(\text{Isom}(\mathbb{H}^{n});\mathbb{R})=0$.
 \end{proof}
The proof above only uses the following two facts:
\begin{enumerate}
\item An isometry of $\mathbb{H}^n$ can be extended to an isometry of $\mathbb{H}^{n+1}$. 
\item If $x_0,\dots,x_k\in\mathbb{H}^{n+1}$ where $k+1\leq n$ then there exists a $g\in\text{Isom}(\mathbb{H}^{n+1})$ such that $gH(x_0,\dots,x_k)\subset \mathbb{H}^n$ where $H(x_0,\dots,x_k)$ denotes the linear subspace spanned by the points $x_0,\dots,x_k$.
\end{enumerate}
These also hold in complex hyperbolic space $\mathbb{H}_{\mathbb{C}}^n$. Hence we immediately obtain the following lemma. 
\begin{Le}
\label{nul}
Let $k\leq n$ and suppose that $H_{c,b}^k(\mathrm{Isom}(\mathbb{H}_{(\mathbb{C})}^n);\mathbb{R})=0$. Then
 \[
H^k_{c,b}(\mathrm{Isom}(\mathbb{H}_{(\mathbb{C})}^{n+1});\mathbb{R})=0.
\] 
\end{Le}
Theorem  \ref{stability} is proved similarly. On cochains define 
\[
j:C((\mathbb{H}_{(\mathbb{C})}^n)^{k+1};\mathbb{R})^{\text{Isom}(\mathbb{H}_{(\mathbb{C})}^n)}\to C((\mathbb{H}_{(\mathbb{C})}^{n+1})^{k+1};\mathbb{R})^{\text{Isom}(\mathbb{H}_{(\mathbb{C})}^{n+1})}
\]
by
\[
j(\beta)(x_0,\dots,x_k):=\beta(gx_0,\dots,gx_k),
\]
where $g\in\text{Isom}(\mathbb{H}_{(\mathbb{C})}^n)$ is such that $g(H(x_0,\dots,x_k))\subset \mathbb{H}_{(\mathbb{C})}^n$. Then 
\[
\delta( j(\beta))(x_0,\dots,x_{k+1})=\delta\beta(g'x_0,\dots,g'x_{k+1}),
\]
with $g'\in \text{Isom}(\mathbb{H}_{(\mathbb{C})}^{n+1})$ such that $g'(H(x_0,\dots,x_{k+1}))\subset \mathbb{H}_{(\mathbb{C})}^n$. Note that such a $g'$ exists since $k+1\leq n$. It follows that if $\beta$ is a cocycle then $j(\beta)$ is as well. Let furthermore 
\[
r:C((\mathbb{H}_{(\mathbb{C})}^{n+1})^{k+1};\mathbb{R})^{\text{Isom}(\mathbb{H}_{(\mathbb{C})}^{n+1})}\to C((\mathbb{H}_{(\mathbb{C})}^n)^{k+1};\mathbb{R})^{\text{Isom}(\mathbb{H}_{(\mathbb{C})}^n)}
\]
be the map defined by restricting a $k$-cochain to $({\mathbb{H}_{(\mathbb{C})}^n})^{k+1}$. Then for a cocycle $\beta$ in degree $k$
\begin{eqnarray*}
r\circ j(\beta)(x_0,\dots,x_k)&=& j(\beta) (x_0,\dots,x_k)\\
&=& \beta(x_0,\dots,x_k),
\end{eqnarray*}
for all $x_0,\dots,x_k\in \mathbb{H}_{(\mathbb{C})}^n$. It follows that $j$ induces an injective map on cohomology.


\addcontentsline{toc}{section}{Bibliography}
\begin{thebibliography}{widest-label}
\bibitem{AM} Austin, T., Moore, C.C.: Continuity Properties of Measurable Group Cohomology. Math. Ann. \textbf{356} (3) (2013),, 885-937.
\bibitem{Bl} Bloch, S.J.: Higher Regulators, Algebraic K-Theory, and Zeta Functions of Elliptic Curves. CRM Monograph Series, vol. 11, Amer. Math. Soc., Providence, RI 2000.
\bibitem{BW} Borel, A., Wallach, N.: Continuous Cohomology, Discrete Subgroups, and Representations of Reductive Groups. Mathematical Surveys and Monographs, vol. 67. Amer. Math. Soc., Providence, RI, second edition 2000.
\bibitem{BT} Bott, R., Tu, L.W.: Differential forms in algebraic topology. Graduate Texts in Mathematics, vol. 82. Springer-Verlag, New York 1982.
\bibitem{Br} G.E. Bredon, Topology and Geometry, Graduate Texts in Mathematics 139, Springer 1993.
\bibitem{BBI} Bucher, M., Burger, M., Iozzi, A.: A dual interpretation of the Gromov-Thurston proof of Mostow rigidity and volume rigidity for representations of hyperbolic lattices. In: Trends in harmonic analysis, Springer INdAM Ser. \textbf{3}, Springer, Milan 2013, 47-76. 
\bibitem{Bu} Buchsbaum, D.: Satellites and universal functors. Ann. Math. (2) \textbf{71} (1960), 199-209.
\bibitem{BM} Burger, M., Monod, N.: Bounded cohomology of lattices in higher rank Lie groups. J. Eur. Math Soc. \textbf{1}(2) (1999), 199-235.
\bibitem{BM2} Burger, M., Monod, N.: On and around the bounded cohomology of $\mathrm{SL}_2$. In: Rigidity in dynamics and geometry, pp. 19-37. Springer, Berlin 2002.
\bibitem{Du} Dupont, J. L.: Bounds for characteristic numbers of flat bundles. In: Algebraic topology, Aarhus 1978  Lecture Notes in Mathematics \textbf{763} Springer, Berlin 1979 , 109-119.
\bibitem{Fre} Fremlin, D.H.: Measure Theory, Volume 4: Topological Measure Theory. Torres Fremlin, Colchester 2005.
\bibitem{Go1} Goncharov, A.B.: Explicit Construction of Characteristic Classes. Adv. in Soviet Math. \textbf{16}(1) (1993), 169-210 .
\bibitem{Go2} Goncharov, A.B.: Geometry of configurations, polylogarithms, and motivic cohomology. Adv. Math. \textbf{114}(2) (1995), 197-318.
\bibitem{Gui} Guichardet, A.: Cohomologie des groupes topologiques et des alg\`ebres de Lie. Textes
Math\'ematiques [Mathematical Texts], vol. 2. CEDIC, Paris 1980.
\bibitem{HO} Hartnick, T., Ott, A.: Bounded cohomology via partial differential equations, I. Geom. Topol. \textbf{19}(6) (2015), 3603-3643.
\bibitem{Mon} Monod, N.: Continuous bounded cohomology of locally compact groups. Lecture Notes in Mathematics \textbf{1758}. Springer-Verlag, Berlin 2001
\bibitem{Mon2} Monod, N.: An invitation to bounded cohomology.  International Congress of Mathematicians, vol. II,  Eur. Math. Soc., Z\"urich 2006, 118-211.
\bibitem{Mon3}Monod, N.: Stabilization for $\mathrm{SL_n}$ in bounded cohomology. In Discrete geometric analysis, Contemp. Math. \textbf{347} Amer. Math. Soc., Providence, RI 2004, 191-202.
\bibitem{Mo12} Moore, C.C.: Extensions and low dimensional cohomology theory of locally compact groups. I, II. Trans. Am. Math. Soc. \textbf{113} (1964), 40-63 .
\bibitem{Mo3} Moore, C.C.: Group Extensions and Cohomology for Locally Compact Groups. III. Trans. Am. Math. Soc. \textbf{221}(1)  (1976), 1-33.
\bibitem{Mo4} Moore, C.C.: Group Extensions and Cohomology for Locally Compact Groups. IV. Trans. Am. Math. Soc. \textbf{221}(1)  (1976), 35-58.
\end{thebibliography}
\end{document}